\documentclass[12pt,a4paper,leqno]{amsart}

\title[Embeddings of $\alpha$-modulation spaces]{Embeddings of $\alpha$-modulation spaces}

\author[J. Toft]{Joachim Toft}

\address{School of Computer Science, Physics and Mathematics, Linn\ae us University, SE--351 95 V\" axj\" o, Sweden.}

\email{joachim.toft@lnu.se}

\author[P. Wahlberg]{Patrik Wahlberg}

\address{Department of Mathematics, University of Turin, Via Carlo Alberto 10, 10123 Torino (TO), Italy.}

\email{patrik.wahlberg@unito.it}


\usepackage{latexsym}
\usepackage{amsmath}
\usepackage{amssymb}
\usepackage{amsfonts}
\usepackage{mathrsfs}
\usepackage{calc}
\usepackage{cite}

\numberwithin{equation}{section}          

\newtheorem{thm}{Theorem}
\numberwithin{thm}{section}

\newcommand{\rubrik}{}
\newtheorem{prop}[thm]{Proposition}
\newtheorem{cor}[thm]{Corollary}
\newtheorem{lem}[thm]{Lemma}

\theoremstyle{definition}

\newtheorem{defn}[thm]{Definition}

\theoremstyle{remark}

\newtheorem{rem}[thm]{Remark}              

\newcommand{\ro}{\mathbb R}

\newcommand{\noo}{\mathbb N_0}
\newcommand{\rr}[1]{\mathbb R^{#1}}

\newcommand{\non}[1]{\mathbb N_0^{#1}}
\newcommand{\zz}[1]{\mathbb Z^{#1}}

\newcommand{\nm}[2]{\Vert #1\Vert _{#2}}

\newcommand{\sets}[2]{\{ \, #1\, ;\, #2\, \} }
\newcommand{\ep}{\varepsilon}
\newcommand{\fy}{\varphi}

\newcommand{\supp}{\operatorname{supp}}

\newcommand{\eabs}[1]{\langle #1\rangle}

\newcommand{\vrum}{\vspace{0.1cm}}

\newcommand{\wt}{\widetilde}
\newcommand{\wh}{\widehat}

\begin{document}

\begin{abstract}
We show upper and lower embeddings of $\alpha_1$-modulation spaces in $\alpha_2$-modulation spaces
for $0 \leq \alpha_1 \leq \alpha_2 \leq 1$, and prove partial results on the sharpness of the embeddings.
\end{abstract}

\maketitle

\emph{Dedicated to Professor Petar Popivanov on the occasion of his 65th birthday}

\section{Introduction}

Let $1 \leq p,q \leq \infty$ and define the indices
\begin{equation}\nonumber
\begin{aligned}
\theta_1(p,q) & = \max \left( 0,q^{-1} - \min( p^{-1},p'^{-1})  \right), \\
\theta_2(p,q) & = \min \left( 0,q^{-1} - \max( p^{-1},p'^{-1})  \right).
\end{aligned}
\end{equation}
Our main result is the following.
For $0 \leq \alpha_1 \leq \alpha_2 \leq 1$, $p,q \in [1,\infty]$ and $s \in \ro$,
we have the embeddings for $\alpha$-modulation spaces
\begin{equation}\label{huvudresultat}
M_{\alpha_2,s + d(\alpha_2-\alpha_1)\theta_1(p,q) }^{p,q}(\rr d)
\subseteq M_{\alpha_1,s}^{p,q} (\rr d)
\subseteq M_{\alpha_2,s + d (\alpha_2-\alpha_1) \theta_2(p,q)}^{p,q}(\rr d).
\end{equation}
(See Theorem \ref{alphaembedding}.)
The embeddings \eqref{huvudresultat} contain known results for embeddings of modulation spaces in Besov spaces \cite{Toft2} and sharpen Gr\"obner's embeddings \cite{Grobner1}.

We also show the sharpness of the embeddings \eqref{huvudresultat} in the following sense.
(See Corollary \ref{sharpnessresult}.)
If $p \geq \min(2,q)$ then
\begin{equation}\label{sharpness1}
M^{p,q}_{\alpha _1,s} \subseteq M^{p,q}_{\alpha _2,t} \quad \Longrightarrow \quad t \leq s + d(\alpha_2-\alpha_1) \theta_2(p,q).
\end{equation}
If $p \leq \max(2,q)$ then
\begin{equation}\label{sharpness2}
M^{p,q}_{\alpha _2,t} \subseteq M^{p,q}_{\alpha _1,s} \quad \Longrightarrow \quad t \geq s + d(\alpha_2-\alpha_1) \theta_1(p,q).
\end{equation}
For $p < \min(2,q)$ we are unable to show the implication \eqref{sharpness1}.
Nevertheless, we conjecture that the implication \eqref{sharpness1} holds also for $p < \min(2,q)$.
By duality, this is equivalent to \eqref{sharpness2} for $p > \max(2,q)$.

\begin{rem}
After finalizing the proof of \eqref{huvudresultat}, we noticed the preprint \cite{Hanwang1} by Han and Wang.
Their results \cite[Theorems 5.1 and 5.2]{Hanwang1} generalize our Theorem \ref{alphaembedding}, and show that
the embeddings \eqref{huvudresultat} hold for all $p,q \in (0,\infty]$, $0 \leq \alpha_1 \leq \alpha_2 \leq 1$ and $s \in \ro$.
This paper provides an alternative proof to Han and Wang's proof in the case $p,q \in [1,\infty]$, and establishes the partial sharpness of the embeddings (sharpness results are not treated in \cite{Hanwang1}).
\footnote{Note added in proof. In an updated version of their manuscript [10], Han and Wang establish the sharpness of the embeddings in all cases.}
\end{rem}

\section{Preliminaries}

$\noo$ denotes the nonnegative integers.
Inclusions $A \subseteq B$ and equalities $A=B$ of topological spaces $A$, $B$,
are understood as embeddings, that is an inclusion is continuous.
We use the standard notations $\mathscr S(\rr d)$, $\mathscr S'(\rr d)$, $C_c^\infty(\rr d)$ for function and distribution spaces (see e.g. \cite{Horm1}).
The Fourier transform of $f \in \mathscr S(\rr d)$ is defined by
$$
\mathscr F f(\xi) = \wh f(\xi) = (2 \pi)^{-d/2} \int_{\rr d} f(x) e^{- i x \cdot \xi} dx.
$$
A Fourier multiplier operator is defined by $\varphi(D)f = \mathscr F^{-1} (\varphi \wh f)$, provided $\varphi$ and $f$ are objects such that the expression makes sense.
For $s \in \ro$ the Sobolev space $H_s(\rr d)$ is defined as the subspace of $f \in \mathscr S'(\rr d)$ such that $\wh f \in L_{\rm loc}^2(\rr d)$ and
$$
\| f \|_{H_s} = \left( \int_{\rr d} \eabs{\xi}^{2s} | \wh f(\xi)|^2 d \xi  \right)^{1/2} < \infty
$$
where $\eabs{\xi}=(1+|\xi|^2)^{1/2}$.

We denote by $|A|$ the cardinality of a finite set $A$, and by $\mu(A)$ the Lebesgue measure of a measurable set $A \subseteq \rr d$.
A closed ball in $\rr d$ of center $a \in \rr d$ and radius $r \geq 0$ is denoted $B(a,r) = \{ x \in \rr d: |x-a| \leq r \}$.
A closed cube in $\rr d$ of center $c$ and side length $2 r$ is denoted $Q(c,r) = \{ x \in \rr d: \max_{1 \leq j \leq d} |x_j-c_j| \leq r \}$.
The conjugate exponent to $p \in [1,\infty]$ is denoted $p'$ and defined by $1/p+1/p'=1$.
The notation $X \lesssim Y$ means that $X \leq C Y$ for some constant $C>0$,
and $X_i \lesssim Y_j$ for $i \in I$ and $j \in J$ means that the constant is uniformly bounded over the index sets $I$ and $J$.
If $X \lesssim Y$ and $Y \lesssim X$ then we write $X \asymp Y$.
Coordinate reflection is denoted $\check f(x) = f(-x)$.

\subsection{Besov spaces}

Define
\begin{equation}\label{Djdef}
D_j = \{ \xi \in \rr d: 2^{j-2} \leq |\xi| \leq 2^j \}, \quad j \geq 1.
\end{equation}
Let $\{ \varphi_j \}_{j=0}^\infty \subseteq C_c^\infty(\rr d)$ be a sequence with the following properties \cite{Bergh1}.
\begin{equation}\label{besovpartition1}
\begin{aligned}
& \supp \varphi_0 \subseteq B(0,1), \\
& \supp \varphi_j \subseteq D_j, \quad j \geq 1, \\
& \sum_{j=0}^\infty \varphi_j(\xi) = 1 \quad \forall \xi \in \rr d.
\end{aligned}
\end{equation}
Then we have for $j \geq 0$
\begin{equation}\label{partition1}
2^{j-1} \leq |\xi| \leq 2^j \quad \Rightarrow \quad \varphi_j(\xi) + \varphi_{j+1}(\xi) = 1.
\end{equation}
The functions $\varphi_j$ for $j \geq 1$ are constructed as dilations $\varphi_j(\xi) = \varphi(2^{1-j} \xi)$ for
a function $\varphi \in C_c^\infty(\rr d)$ supported in $D_1$ (cf. \cite{Bergh1}).
Let $p,q \in [1,\infty]$ and let $s \in \ro$.
The Besov space $B_s^{p,q}(\rr d)$ is defined as the space of all $f \in \mathscr S'(\rr d)$ such that
\begin{equation}\label{besovspacenorm1}
\| f \|_{B_s^{p,q}} = \left( \sum_{j=0}^\infty \left( 2^{j s} \| \varphi_j(D) f\|_{L^p} \right)^q \right)^{1/q} < \infty
\end{equation}
when $q<\infty$ and with the standard modification when $q=\infty$ \cite{Bergh1}. We abbreviate $B_s^{p,p}=B_s^{p}$ and $B_0^{p,q} =B^{p,q}$.

\subsection{$\alpha$-modulation spaces}\label{alphasection}

We need the following definitions introduced by Feichtinger and Gr\"obner \cite{Grobner1,Feichtinger1,Feichtinger2,Feichtinger3} (cf. \cite{Borup1,Fornasier1}).

\begin{defn}
A countable set $\mathcal Q$ of subsets $Q \subseteq \rr d$ is called an admissible covering provided
\begin{align}
& \bigcup_{Q \in \mathcal Q} Q = \rr d, \nonumber \\
& |\{ Q' \in \mathcal Q: Q \cap Q' \neq \emptyset \}| \leq n_0 \quad \forall Q \in \mathcal Q, \label{finiteheight}
\end{align}
for some finite integer $n_0$.
\end{defn}

For each $Q \in \mathcal Q$, let
\begin{align}
r_Q & = \sup \{ r \in \ro: B(c,r) \subseteq Q \ \mbox{for some $c \in \rr d$} \}, \label{innermeasure} \\
R_Q & = \inf \{ R \in \ro: Q \subseteq B(c,R)  \ \mbox{for some $c \in \rr d$} \}. \label{outermeasure}
\end{align}

\begin{defn}\label{alphacovering}
Let $\alpha \in [0,1]$. An admissible covering $\{ Q \}_{Q \in \mathcal Q}$ is called an $\alpha$-covering provided there exists a constant $K \geq 1$ such that
\begin{align}
& \mu(Q) \asymp \eabs{ x }^{\alpha d}, \quad x \in Q, \quad Q \in \mathcal Q, \label{alphacoveringsize1} \\
& R_Q/r_Q \leq K, \quad Q \in \mathcal Q. \label{alphacoveringsize2}
\end{align}
\end{defn}

\begin{defn}
Let $\alpha \in [0,1]$ and let $\{ Q \}_{Q \in \mathcal Q}$ be an $\alpha$-covering
of $\rr d$. Then $\{ \psi_Q \}_{Q \in \mathcal Q}$ is called a bounded admissible
partition of unity corresponding to $\mathcal Q$ ($\mathcal Q$-BAPU) provided
\begin{align}
& \supp \psi_Q \subseteq Q, \quad Q \in \mathcal Q, \nonumber \\
& \sum_{Q \in \mathcal Q} \psi_Q (\xi) = 1 \quad \forall \xi \in \rr d, \nonumber \\
& \sup_{Q \in \mathcal Q} \| \mathscr F \psi_Q \|_{L^1} < \infty. \label{bapu3}
\end{align}
\end{defn}

We will call a $\mathcal Q$-BAPU an $\alpha$-BAPU when $\mathcal Q$ is an $\alpha$-covering.

\begin{defn}
Let $\alpha \in [0,1]$, $p, q \in [1,\infty]$, $s \in \ro$, let $\{ Q \}_{Q \in \mathcal Q}$ be an $\alpha$-covering of $\rr d$ and let $\{ \psi_Q \}_{Q \in \mathcal Q}$ be a $\mathcal Q$-BAPU. The weighted $\alpha$-modulation space $M_{\alpha,s}^{p,q}(\rr d)$ is defined as all $f \in \mathscr S'(\rr d)$ such that
\begin{equation}\label{alphamodulationnorm}
\| f \|_{M_{\alpha,s}^{p,q}} = \left( \sum_{Q \in \mathcal Q} \eabs{\xi_Q}^{q s}
\| \psi_Q(D) f \|_{L^p}^q \right)^{1/q} < \infty
\end{equation}
where $\xi_Q \in Q$ for all $Q \in \mathcal Q$, when $q < \infty$. If $q=\infty$ the global $l^q$ norm in \eqref{alphamodulationnorm} is replaced by $l^\infty$.
\end{defn}

The $\alpha$-modulation spaces contain as extreme cases the
frequency-weighted modulation spaces (cf.
\cite{Feichtinger1,Grochenig1}) $M_s^{p,q} = M_{0,s}^{p,q}$ ($\alpha=0$)
and the Besov spaces $B_s^{p,q} = M_{1,s}^{p,q}$ ($\alpha=1$) (cf. \cite{Grobner1}).
The number $\alpha$ thus parametrizes a scale of spaces that in some sense
is intermediate between the modulation spaces and the Besov spaces.
We abbreviate $M_{\alpha,s}^{p,p} =
M_{\alpha,s}^{p}$, $M_s^{p,p}=M_s^{p}$ and $M_0^{p,q}= M^{p,q}$
(the unweighted or classical modulation spaces).
For $t \geq s$ we have the embedding $M_{\alpha,t}^{p,q} \subseteq M_{\alpha,s}^{p,q}$, $\alpha \in [0,1]$, $p, q \in [1,\infty]$.

For $\alpha$ in the interval $0 \leq \alpha < 1$, that is, excluding the Besov
spaces, we will use the following $\alpha$-covering and an associated
$\mathcal Q$-BAPU (cf. \cite{Borup1}). Set
\begin{equation}\label{borupcovering1}
B_k = B( k |k|^\beta, r|k|^\beta), \quad k \in \zz d \setminus 0,
\end{equation}
where $\beta = \alpha/(1-\alpha)$.
Note that $B_k = B(\xi_k,r|\xi_k|^\alpha)$ where $\xi_k=k |k|^\beta$.
For $r>0$ sufficiently large, $\mathcal
Q = \{ B_k \}_{k \in \zz d \setminus 0}$ is an $\alpha$-covering of $\rr d$
according to \cite[Theorem 2.6]{Borup1}. Moreover, a $\mathcal Q$-BAPU
$\{ \psi_k\}_{k \in \zz d \setminus 0}$ such that $\supp \psi_k \subseteq B_k$
for all $k \in \zz d \setminus 0$ can be constructed (see
\cite[Proposition A.1]{Borup1}).

We will use Borup and Nielsen's Banach frame construction for $M_{\alpha,s}^{p,q}(\rr d)$, based on multivariate brushlet systems (cf. \cite{Borup1}).
Let
\begin{equation}\nonumber
Q_k = Q ( k |k|^\beta, r|k|^\beta), \quad k \in \zz d \setminus 0,
\end{equation}
where again $\beta = \alpha/(1-\alpha)$.
If $r>0$ is sufficiently large then $\mathcal Q = \{ Q_k \}_{k \in \zz d \setminus 0}$ is an $\alpha$-covering of $\rr d$.
One can construct a sequence of functions
\begin{equation}\nonumber
( w_{n,k} )_{n \in \non d, \ k \in \zz d \setminus 0} \subseteq \mathscr S(\rr d)
\end{equation}
such that $( w_{n,k} )_{n \in \non d}$ is an orthonormal system, with $\supp \wh w_{n,k} \subseteq Q_k$,
for each $k \in \zz d \setminus 0$.
Each function $w_{n,k}$ is constructed as a tensor product
\begin{equation}\label{wnkdefinition}
w_{n,k}= \bigotimes_{j=1}^d w_{n_j,I_{k,j}}
\end{equation}
where
$Q_k = \Pi_{j=1}^d I_{k,j}$, whose components are, simplifying notation to $n=n_j$, $I=I_{k,j}$,
\begin{equation}\nonumber
w_{n,I}(x) = \sqrt{\frac{\mu(I)}{2}} \ e^{i a_I x} \ \Big( g(\mu(I) (x+e_{n,I}) + g(\mu(I) (x-e_{n,I}) \Big), \quad x \in \ro,
\end{equation}
where $e_{n,I} = \pi(n+1/2)/\mu(I)$,
$a_I$ denotes the left end point of $I$, i.e. $I=[a_I,b_I]$,
and $g \in \mathscr F C_c^\infty(\ro)$ with $\supp \wh g \subseteq [0,1]$.
For more details about the sequence of functions $( w_{n,k} )_{n \in \non d, \ k \in \zz d \setminus 0}$ we refer to \cite{Borup1}.

Borup and Nielsen \cite{Borup1} show that the sequence $( w_{n,k} )$ is a (quasi-) Banach frame for $M_{\alpha,s}^{p,q}(\rr d)$
for $0 < p,q \leq \infty$ and $s \in \ro$. We restrict our interest to the exponents $p,q \in [1,\infty]$.
Let $p,q \in [1,\infty]$, $s \in \ro$, let $f \in M_{\alpha,s}^{p,q}(\rr d)$, and define the coefficient sequence
\begin{equation}\label{coefficientoperator1}
c_{n,k} = (f, w_{n,k})_{L^2}, \quad n \in \non d, \quad k \in \zz d \setminus 0
\end{equation}
where $w_{n,k}$ is defined by \eqref{wnkdefinition}.
The coefficient operator is defined by $(D f)_{n,k} = c_{n,k}$, $n \in \non d$, $k \in \zz d \setminus 0$.
The Banach frame property means in this case that
\begin{equation}\label{normequivalence1}
\| f \|_{M_{\alpha,s}^{p,q}} \asymp \| c \|_{m_{\alpha,s}^{p,q}},
\end{equation}
where the sequence space $m_{\alpha,s}^{p,q} = m_{\alpha,s}^{p,q}(\non d \times \zz d \setminus 0 )$ is defined by the norm
\begin{equation}\label{sequencespace1}
\| c \|_{m_{\alpha,s}^{p,q}} = \left( \sum_{k \in \zz d \setminus 0} \left( \sum_{n \in \non d}
\left( |k|^{\frac1{1-\alpha} \left( s + \alpha d \left( \frac1{2} - \frac1{p} \right) \right)} |c_{n,k}| \right)^p \right)^{q/p} \right)^{1/q}
\end{equation}
when $p,q < \infty$ and suitably modified otherwise.
Moreover, there exists a reconstruction operator $R$ defined by
$$
R \ c = \sum_{k \in \zz d \setminus 0, \ n \in \non d} c_{n,k} \ \wt w_{n,k},
$$
where $( \wt w_{n,k} )_{k \in \zz d \setminus 0, n \in \non d}$ is a dual frame defined by $\wt w_{n,k} = \psi_k(D) w_{n,k}$,
$n \in \non d$, $k \in \zz d \setminus 0$.
The operator $R$ is bounded as
\begin{equation}\label{reconstruction1}
\| R c \|_{M_{\alpha,s}^{p,q}} \lesssim \| c \|_{m_{\alpha,s}^{p,q}}, \quad c \in m_{\alpha,s}^{p,q},
\end{equation}
and $R D = id_{M_{\alpha,s}^{p,q}}$. These results are proved in \cite[Theorem 4.3]{Borup1}.

Let $\mathscr{M}_{\alpha,s}^{p,q}(\rr d)$ be the completion of $\mathscr{S}(\rr d)$ in the norm $\|
\cdot \|_{M_{\alpha,s}^{p,q}(\rr d)}$. In the next result we collect some important properties of the $\alpha$-modulation spaces. The result is a generalization of the corresponding result for modulation spaces.

\begin{prop}\label{alphaproperties}
Let $\alpha \in [0,1]$, $s \in \ro$ and $p,q \in [1,\infty]$. The following holds.
\begin{enumerate}
\item[{\rm{(i)}}] The space $M_{\alpha,s}^{p,q}(\rr d)$ is a Banach
space which is independent of the sequence $\{ \xi_Q \}_{Q \in \mathcal Q}$ as long as
$\xi_Q \in Q$ for all $Q \in \mathcal Q$, and also independent of the $\alpha$-covering $\{ Q \}_{Q \in \mathcal Q}$ and of the $\mathcal Q$-BAPU $\{ \psi_Q \}_{Q \in \mathcal Q}$. Varying these parameters gives rise to equivalent norms.

\vrum

\item[{\rm{(ii)}}] The $L^2$-product $(\cdot,\cdot)$ on $\mathscr S (\rr {d}) \times \mathscr S (\rr {d})$
extends to a continuous sesquilinear form on $M_{\alpha,s}^{p,q} (\rr {d})\times M_{\alpha,-s}^{p',q'}(\rr
{d})$. Furthermore,
$$
\| f \| = \sup |(f,g)|
$$
with supremum taken over all $g\in \mathscr S (\rr d)$ such that $\| g \|_{M_{\alpha,-s}^{p',q'}} \leq 1$, is a norm equivalent to $\| f \|_{M_{\alpha,s}^{p,q}}$. If $p,q<\infty$, then the dual space of $M_{\alpha,s}^{p,q}$ can be identified with $M_{\alpha,-s}^{p',q'}$ through the form $(\cdot,\cdot)$.

\vrum

\item[{\rm{(iii)}}] Assume that $0\leq \theta \leq 1$, $p,q,p_1,p_2,q_1,q_2 \in [1,\infty]$, $s, s_1, s_2 \in \ro$
satisfy
\begin{equation*}
\frac{1}{p} = \frac{1-\theta}{p_1}+\frac{\theta}{p_2}, \quad
\frac{1}{q} = \frac{1-\theta}{q_1}+\frac{\theta}{q_2}, \quad
s = (1-\theta) s_1 + \theta s_2.
\end{equation*}
Then complex interpolation gives
$$
(\mathscr{M}_{\alpha,s_1}^{p_1,q_1},\mathscr{M}_{\alpha,s_2}^{p_2,q_2})_{[\theta]}= \mathscr{M}_{\alpha,s}^{p,q}.
$$

\vrum

\item[{\rm{(iv)}}] It holds $\mathscr{M}_{\alpha,s}^{p,q} \subseteq
M_{\alpha,s}^{p,q}$ with equality if
$p<\infty$ and $q<\infty$.
\end{enumerate}
\end{prop}

\begin{proof}
(i) See \cite[Theorems 2.2, 2.3 and 3.7]{Feichtinger2} and \cite[Theorem 4.1]{Feichtinger3}.

(ii)
The fact that the dual space of $M_{\alpha,s}^{p,q}$, for $1 \leq p,q<\infty$, can be identified with $M_{\alpha,-s}^{p',q'}$ is a consequence of \cite[Theorem 2.8]{Feichtinger2}.
Let $1 \leq p,q \leq \infty$.
From \cite[Theorem 2.3]{Feichtinger2} it follows
\begin{equation}\nonumber
\left| (f,g) \right| \lesssim \| f \|_{M_{\alpha,s}^{p,q}} \| g \|_{M_{\alpha,-s}^{p',q'}}, \quad g \in \mathscr S(\rr d).
\end{equation}
For the reverse inequality we first let $0 \leq \alpha < 1$.
By \eqref{normequivalence1}
\begin{equation}\nonumber
\| f \|_{M_{\alpha,s}^{p,q}} \lesssim \| c \|_{m_{\alpha,s}^{p,q}},
\end{equation}
where the sequence $c$ is defined by \eqref{coefficientoperator1}.
The $m_{\alpha,s}^{p,q}$-norm of $c$ is the mixed $\ell^{p,q}$ norm of $\omega c$,
where the weight $\omega$ depends on $p,\alpha,s$ as
\begin{equation}\nonumber
\omega_{n,k} = \omega_k = |k|^{\frac1{1-\alpha} \left( s + \alpha d \left( \frac1{2}-\frac1{p} \right) \right)}.
\end{equation}
An application of \cite[Lemma 3.1]{Benedek1} yields
\begin{equation}\nonumber
\| c \|_{m_{\alpha,s}^{p,q}} = \| \omega c \|_{\ell^{p,q}} = \sup | (\omega c,d)_{\ell^2} |
\end{equation}
with supremum taken over all sequences $(d_{n,k})$ of finite support and $\| d \|_{\ell^{p',q'}} \leq 1$.
Let $(d_{n,k})$ be a sequence of finite support such that $\| d \|_{\ell^{p',q'}} \leq 1$ and
\begin{equation}\nonumber
\| \omega c \|_{\ell^{p,q}} \leq 2 | (\omega c,d)_{\ell^2} |,
\end{equation}
and set
\begin{equation}\nonumber
g = \sum_{k \in \zz d \setminus 0} \ \sum_{n \in \non d} \omega_k \ d_{n,k} \ w_{n,k}.
\end{equation}
Then $g \in \mathscr S(\rr d)$ since the sum is finite, and $(f,g) = (\omega c,d)_{\ell^2}$.
The following inequality follows from the proofs of \cite[Lemma 3.2 and Lemma 4.2]{Borup1}. If $p,q \in [1,\infty]$ and $s \in \ro$, then
\begin{equation}\nonumber
\left\| \sum_{k \in \zz d \setminus 0} \ \sum_{n \in \non d} d_{n,k} \ w_{n,k} \right\|_{M_{\alpha,-s}^{p',q'}}
\lesssim \| d \|_{m_{\alpha,-s}^{p',q'}}.
\end{equation}
This gives
\begin{equation}\nonumber
\| g \|_{M_{\alpha,-s}^{p',q'}} \lesssim \| \omega d \|_{m_{\alpha,-s}^{p',q'}}
= \| d \|_{\ell^{p',q'}} \leq 1.
\end{equation}
Hence we have proved that $\| f \|_{M_{\alpha,s}^{p,q}} \lesssim \| f \|$ when $0 \leq \alpha < 1$.

It remains to prove the corresponding inequality when $\alpha=1$, in which case $M_{\alpha,s}^{p,q} = B_s^{p,q}$.
Let $\{ \varphi_j \}_{j=0}^\infty \subseteq C_c^\infty(\rr d)$ be a sequence that satisfies \eqref{besovpartition1}
and $\varphi_j (\xi) = \varphi(2^{1-j} \xi)$ for $j \geq 1$ where $\varphi \in C_c^\infty(\rr d)$ and $\supp \varphi \subseteq D_1$.
The $B_s^{p,q}$-norm defined by \eqref{besovspacenorm1} is the mixed Lebesgue norm $L^{p,q}(\rr d \times \noo)$,
where $\rr d$ is equipped with the Lebesgue measure and $\noo$ with the counting measure,
of the function $F(x,j) = 2^{js} \varphi_j(D) f(x)$.
According to \cite[Lemma 3.1]{Benedek1} we have
\begin{equation}\nonumber
\| f \|_{B_s^{p,q}} = \sup \left| \sum_{j=0}^\infty 2^{js} (\varphi_j(D) f,g_j)_{L^2} \right|
\end{equation}
where the supremum is taken over all sequences $(g_j)_{0}^\infty$ of simple functions of compact support
$g_j$ such that $g_j \equiv 0$ for $j > N$ for some $N \geq 0$, and
\begin{equation}\nonumber
\left( \sum_{j=0}^\infty \| g_j \|_{L^{p'}}^{q'} \right)^{1/q'} \leq 1
\end{equation}
if $q'<\infty$, and $\sup_{0 \leq j < \infty} \| g_j \|_{L^{p'}} \leq 1$ if $q'=\infty$.
Therefore there exists $N \geq 0$ and $(g_j)_{0}^N \subseteq L^{p'}(\rr d)$ such that
\begin{equation}\nonumber
\| f \|_{B_s^{p,q}} \leq 2 \sum_{j=0}^N 2^{js} (\varphi_j(D) f,g_j)_{L^2}
= 2( f,\sum_{j=0}^N 2^{js} \varphi_j(D) g_j)_{L^2}
\end{equation}
and
\begin{equation}\label{unitball1}
\left( \sum_{j=0}^N \| g_j \|_{L^{p'}}^{q'} \right)^{1/q'} \leq 1
\end{equation}
(modified as above if $q'=\infty$).
Set
$$
g = \sum_{j=0}^N 2^{js} \varphi_j(D) g_j \in \mathscr S(\rr d).
$$
We have $\sup_{j \geq 0} \| \mathscr F^{-1} \varphi_j \|_{L^1} \lesssim 1$.
By means of \eqref{partition1} and Young's inequality, we obtain for $k \geq 1$
\begin{equation}\nonumber
\begin{aligned}
\| \varphi_k(D) g \|_{L^{p'}}
& = \left\| \sum_{j = k-1}^{\min (N,k+1)} 2^{js} \varphi_k(D)  \varphi_j(D) g_j \right\|_{L^{p'}} \\
& \lesssim
2^{(k-1)s} \left\| g_{k-1} \right\|_{L^{p'}}
+ 2^{ks} \left\| g_{k} \right\|_{L^{p'}}
+ 2^{(k+1)s} \left\| g_{k+1} \right\|_{L^{p'}},
\end{aligned}
\end{equation}
and
\begin{equation}\nonumber
\begin{aligned}
\| \varphi_0(D) g \|_{L^{p'}}
& = \left\| \sum_{j = 0}^{\min (N,1)} 2^{js} \varphi_0(D)  \varphi_j(D) g_j \right\|_{L^{p'}} \\
& \lesssim
\left\| g_{0} \right\|_{L^{p'}}
+ 2^s \left\| g_{1} \right\|_{L^{p'}},
\end{aligned}
\end{equation}
which gives, by means of \eqref{unitball1}, $\| g \|_{B_{-s}^{p',q'}} \lesssim 1$.
It follows that $\| f \|_{M_{s,1}^{p,q}} \lesssim \| f \|$.

(iii) This follows from \cite[Corollary 2.4]{Feichtinger2} (cf. \cite[Bemerkung F.2]{Grobner1}).

(iv) See \cite[Theorem 2.2]{Feichtinger2}.
\end{proof}

\section{Embeddings of $\alpha$-modulation spaces}

We need the following elementary lemma (cf. \cite[Prop.~2.5]{Hanwang1} and \cite{Grobner1}), a proof of which is provided as a service to the reader.

\begin{lem}\label{alphasobolev}
If $\alpha \in [0,1]$ and $s \in \ro$ then
$M_{\alpha,s}^{2}(\rr d) = H_s(\rr d)$.
\end{lem}

\begin{proof}
For the Besov space case ($\alpha=1$) the result $B_s^2(\rr d) = H_s(\rr d)$ is well known (see e.g. \cite[Theorem 6.4.4]{Bergh1}). Let $0 \leq \alpha < 1$.
We use the $\alpha$-covering \eqref{borupcovering1} $\{ B_k \}_{k \in \zz d \setminus 0}$
for $r>0$ sufficiently large, and an associated BAPU $\{ \psi_k\}_{k \in \zz d \setminus 0}$ such that $0 \leq \psi_k \leq 1$ for all $k \in \zz d \setminus 0$.
Parseval's formula and \eqref{alphacoveringsize1} yield
\begin{equation}\nonumber
\begin{aligned}
\| f \|_{M_{\alpha,s}^{2}(\rr d)}^2 & = \sum_{k \in \zz d \setminus 0} \eabs{\xi_k}^{2s}
\int_{B_k} \psi_k(\xi)^2 |\wh f(\xi)|^2 d \xi
\\
& \lesssim \sum_{k \in \zz d \setminus 0}   \int_{B_k} \psi_k(\xi) \eabs{\xi}^{2s} |\wh f(\xi)|^2 d \xi
= \| f \|_{H_s(\rr d)}^2,
\end{aligned}
\end{equation}
i.e. $H_s \subseteq M_{\alpha,s}^2$. For the opposite inclusion, we note that
\begin{equation}\label{lowerbound1}
\sum_{k \in \zz d \setminus 0} \psi_k(\xi)^2 \geq C, \quad \xi \in \rr d,
\end{equation}
holds for some $C>0$. In fact, if this would not the case, then for any $\ep>0$ there exists $\xi \in \rr d$ such that
\begin{equation}\nonumber
\sum_{k \in \zz d \setminus 0} \psi_k(\xi)^2 < \ep.
\end{equation}
Let $\ep < n_0^{-2}$ where $n_0$ is the upper bound \eqref{finiteheight} corresponding to the covering $\{ B_k \}_{k \in \zz d \setminus 0}$,
and let $\xi \in \rr d$ denote the corresponding vector.
Then $\psi_k(\xi) < \sqrt{\ep}$ for all $k \in \zz d \setminus 0$.
Since $\xi \in B_j$ for some $j \in \zz d \setminus 0$ we obtain from
\eqref{finiteheight}
$$
\sum_{k \in \zz d \setminus 0} \psi_k(\xi) = \sum_{k: \ B_k \cap B_j \neq \emptyset}
\psi_k(\xi) < n_0 \sqrt{\ep} < 1
$$
which is a contradiction. Thus \eqref{lowerbound1} holds for some $C>0$.

By means of \eqref{lowerbound1} and again \eqref{alphacoveringsize1} we obtain
\begin{equation}\nonumber
\begin{aligned}
\| f \|_{H_s(\rr d)}^2 & \leq C^{-1} \int_{\rr d} \sum_{k \in \zz d \setminus 0} \psi_k(\xi)^2 \eabs{\xi}^{2s}  |\wh f(\xi)|^2 d \xi \\
& \lesssim \sum_{k \in \zz d \setminus 0} \eabs{\xi_k}^{2s} \int_{B_k} \psi_k(\xi)^2 |\wh f(\xi)|^2 d \xi \\
& = \| f \|_{M_{\alpha,s}^{2}(\rr d)}^2,
\end{aligned}
\end{equation}
i.e. $M_{\alpha,s}^2 \subseteq H_s$ and the proof is complete.
\end{proof}

Embeddings for $\alpha$-modulation spaces have been proved by Gr\"obner
\cite{Grobner1}, Han and Wang \cite{Hanwang1}, and, for the modulation space case $\alpha=0$,
by Okoudjou \cite{Okoudjou1} and the first named author of this
article \cite{Toft1,Toft2}.

The result \cite[Theorem 2.10]{Toft2} imply the embeddings, for $p,q \in [1,\infty]$
and $s \in \ro$,
\begin{equation}\label{toftembedding1}
B_{s + d \theta_1(p,q)}^{p,q}(\rr d) \subseteq M_{0,s}^{p,q} (\rr d) \subseteq B_{s + d \theta_2(p,q)}^{p,q}(\rr d).
\end{equation}
Here the indices $\theta_1$ and $\theta_2$ are defined by
\begin{equation}\label{indices1}
\begin{aligned}
\theta_1(p,q) & = \max \left( 0,q^{-1} - \min( p^{-1},p'^{-1})  \right), \\
\theta_2(p,q) & = \min \left( 0,q^{-1} - \max( p^{-1},p'^{-1})  \right) = - \theta_1(p',q').
\end{aligned}
\end{equation}
The unweighted versions (i.e. $s=0$) of these embeddings were proved in \cite[Theorem 3.1]{Toft1}.
They imply the embeddings, for $p,q \in [1,\infty]$,
\begin{equation}\label{toftembedding2}
B_{d \theta_1(p,q)}^{p,q}(\rr d) \subseteq M^{p,q} (\rr d) \subseteq B_{d \theta_2(p,q)}^{p,q}(\rr d),
\end{equation}
and they have been proven to be sharp.
The sharpness was obtained independently by Huang and Wang \cite[Theorem 1.1]{Wang1},
and by Sugimoto and Tomita \cite[Theorem 1.2]{Sugimoto1}, and means the following.
If $p,q \in [1,\infty]$ and $B_s^{p,q}(\rr d) \subseteq M^{p,q}(\rr d)$ then $s \geq d \theta_1(p,q)$.
If $p,q \in [1,\infty]$ and $M^{p,q} (\rr d) \subseteq B_s^{p,q}(\rr d)$ then $s \leq d \theta_2(p,q)$.
(By duality, the two assertions are equivalent.)
This gives the sharpness also for the weighted case \eqref{toftembedding1},
since $\eabs{D}^t$ is a homeomorphism
$B_{s}^{p,q} \mapsto B_{s-t}^{p,q}$ for any $t,s \in \ro$ (cf. \cite{Bergh1}) as well as
$M_{0,s}^{p,q} \mapsto M_{0,s-t}^{p,q}$ for any $t,s \in \ro$ (cf. \cite[Cor.~2.3]{Toft2}).
The sharpness of \eqref{toftembedding1} reads:
\begin{equation}\nonumber
\begin{aligned}
B_t^{p,q}(\rr d) \subseteq M_{0,s}^{p,q}(\rr d) \quad & \Longrightarrow \quad t \geq s + d \theta_1(p,q), \quad p,q \in [1,\infty], \\
M_{0,s}^{p,q}(\rr d) \subseteq B_t^{p,q}(\rr d) \quad & \Longrightarrow \quad t \leq s + d \theta_2(p,q), \quad p,q \in [1,\infty].
\end{aligned}
\end{equation}

Note that the embeddings \eqref{toftembedding1} and \eqref{toftembedding2}
are restricted to upper and lower embeddings of $0$-modulation spaces in $1$-modulation spaces,
and give no information on upper and lower embeddings of $M_{\alpha_1,s}^{p,q}$
in $M_{\alpha_2,t}^{p,q}$ for general $\alpha_1,\alpha_2 \in [0,1]$.

Gr\"obner's embeddings \cite[Theorems F.6, F.7 and pp.~66--68]{Grobner1} reads
\begin{equation}\label{grobnerembedding1}
M_{\alpha_2,s + d(\alpha_2-\alpha_1)\nu_1(p,q) }^{p,q}(\rr d)
\subseteq M_{\alpha_1,s}^{p,q} (\rr d)
\subseteq M_{\alpha_2,s + d (\alpha_2-\alpha_1) \nu_2(p,q)}^{p,q}(\rr d),
\end{equation}
for $0 \leq \alpha_1 \leq \alpha_2 \leq 1$, $p,q \in [1,\infty]$ and $s \in \ro$, where the indices $\nu_1$ and $\nu_2$ are defined by
\begin{equation}\label{grobnerindices1}
\begin{aligned}
\nu_1(p,q) & = \theta_1(p,q) + \max \left( 0,q^{-1} - \max( p^{-1},p'^{-1})  \right), \\
\nu_2(p,q) & = \theta_2(p,q) + \min \left( 0,q^{-1} - \min( p^{-1},p'^{-1})  \right) = - \nu_1(p',q').
\end{aligned}
\end{equation}
Since $\nu_1(p,q) \geq \theta_1(p,q)$ and $\nu_2(p,q) \leq \theta_2(p,q)$,
the embeddings \eqref{toftembedding1} improve Gr\"obner's embeddings \eqref{grobnerembedding1}
when $\alpha_1=0$ and $\alpha_2=1$.

\par

We are now in a position to present our main embedding theorem, which is both a sharpening of \eqref{grobnerembedding1} and a generalization of \eqref{toftembedding1} to general $\alpha$-modulation spaces.
In the proof of the theorem we need the following lemma.

\begin{lem}\label{countinglemma}
Suppose $0 \leq \alpha_1 \leq \alpha_2 \leq 1$, $\{ Q_j \}_{j \in J}$ is an $\alpha_1$-covering,
$\{ P_i \}_{i \in I}$ is an $\alpha_2$-covering, and let $\eta _j \in Q_j$ for all $j \in J$, and $\xi _i \in P_i$ for all $i \in I$.
If
\begin{equation}\nonumber
\begin{aligned}
\Omega _i & = \sets {j \in J}{Q_j \cap P_i \neq \emptyset}, \quad i \in I, \\
\Lambda _j & = \sets {i \in I}{Q_j \cap P_i \neq \emptyset}, \quad j \in J,
\end{aligned}
\end{equation}
then
\begin{align}
|\Omega _i| & \lesssim \eabs {\xi _i}^{d(\alpha _2-\alpha _1)}, &  i \in I, \label{sizesubset1} \\
|\Lambda _j| & \lesssim 1, & j \in J, \label{sizesubset2}
\end{align}
and $\eabs{\xi_i} \asymp \eabs{\eta_j}$ for $j \in \Omega_i$ for all $i \in I$, and for $i \in \Lambda_j$ for all $j \in J$.
\end{lem}

\begin{proof}
By the ``disjointization lemma'' \cite[Lemma~2.9]{Feichtinger2}, for any admissible covering $\{ Q_j \}_{j \in J}$ we can split the index set as $J = \bigcup_{k=1}^{n_0} J_k$, where $n_0$ is finite, $\{ J_k \}$ are pairwise disjoint, and
$j, j' \in J_k$, $j \neq j'$ imply $Q_j \cap Q_{j'} = \emptyset$ for $1 \leq k \leq n_0$.

Let $i \in I$. By \eqref{alphacoveringsize1} we have $\mu(Q_j) \asymp \eabs{\xi_i}^{d \alpha_1}$ for all $j \in \Omega_i$.
By \eqref{outermeasure} and \eqref{alphacoveringsize2} we have $P_i \subseteq B(c_i,2 R_2)$ where $R_2^d \lesssim \mu(P_i)$, for some $c_i \in \rr d$.
Let $j \in \Omega_i$ and $x_j \in Q_j \cap P_i$.
Again \eqref{outermeasure}, \eqref{alphacoveringsize1}, \eqref{alphacoveringsize2}
give $Q_j \subseteq B(b_j,2 R_1)$ where $R_1^d \lesssim \eabs{x_j}^{d \alpha_1} \lesssim \eabs{x_j}^{d \alpha_2} \lesssim \mu(P_i) \lesssim R_2^d$, for some $b_j \in \rr d$.
It follows that $Q_j \subseteq B(c_i,C R_2)$ for some $C>0$.
Combining these observations, we obtain for $1 \leq k \leq n_0$
$$
\eabs{\xi_i}^{d \alpha_1} |\Omega _i \cap J_k| \asymp \sum_{j \in \Omega _i \cap J_k} \mu(Q_j)
\leq \mu( B(c_i,C R_2 ) \lesssim \eabs{\xi_i}^{d \alpha_2},
$$
whereupon \eqref{sizesubset1} follows from the disjointization lemma.
The proof of \eqref{sizesubset2} is similar.
The final statement of the lemma is a direct consequence of \eqref{alphacoveringsize1}.
\end{proof}

\begin{thm}\label{alphaembedding}
Let $p,q \in [1,\infty]$, $s \in \ro$ and $0 \leq \alpha_1 \leq \alpha_2 \leq 1$. Then
\begin{equation}\label{newembedding1}
M_{\alpha_2,s + d(\alpha_2-\alpha_1)\theta_1(p,q) }^{p,q}(\rr d)
\subseteq M_{\alpha_1,s}^{p,q} (\rr d)
\subseteq M_{\alpha_2,s + d (\alpha_2-\alpha_1) \theta_2(p,q)}^{p,q}(\rr d),
\end{equation}
and, for some constant $C>0$, it holds for $f\in \mathscr S'(\rr d)$
\begin{equation}\nonumber
C^{-1}\nm f{M^{p,q}_{\alpha _2,s+d(\alpha _2-\alpha _1)\theta _2(p,q)}}
\le \nm f{M^{p,q}_{\alpha _1,s}} \le C \nm f{M^{p,q}_{\alpha _2,s+d(\alpha _2-\alpha _1)
\theta _1(p,q)}}.
\end{equation}
\end{thm}

\begin{proof}
By duality it suffices to prove the right hand side embedding.
Let $s \in \ro$, let $\{ \fy _j \}$ be an $\alpha _1$-BAPU such that $\fy _j \geq 0$ for all $j$, let $\{ \psi _i \}$ be an
$\alpha _2$-BAPU such that $\psi _i \geq 0$ for all $i$, let $\eta_j \in \supp \varphi _j$ for all $j$, and let $\xi_i \in \supp \psi _i$ for all $i$.
If
\begin{align}
\Omega_i & = \sets {j}{\supp \fy _j \cap \supp \psi _i\neq \emptyset} \nonumber \\
\Lambda_j & = \sets {i}{\supp \fy _j \cap \supp \psi _i\neq \emptyset} \nonumber \\
\end{align}
then by Lemma \ref{countinglemma}
\begin{align}
|\Omega _i| & \lesssim \eabs {\xi _i}^{d(\alpha _2-\alpha _1)} & \mbox{for all} \quad i, \nonumber \\
|\Lambda_j| & \lesssim 1 & \mbox{for all} \quad j, \nonumber
\end{align}
and $\eabs{\xi_i} \asymp \eabs{\eta_j}$ for $j \in \Omega_i$ for all $i$, and for $i \in \Lambda_j$ for all $j$.
This gives, using \eqref{lowerbound1},
\begin{equation}\nonumber
\begin{aligned}
\nm {\psi _i(D)f}{L^2}^2\eabs {\xi _i}^{2 s - d(\alpha _2-\alpha _1)}
& =
\nm {\psi _i\widehat f}{L^2}^2\eabs {\xi _i}^{2 s - d(\alpha _2-\alpha _1)}
\\
& \lesssim \sum _{j\in \Omega _i}\int \fy _j^2(\xi )\psi _i^2(\xi )|\widehat f(\xi )|^2\, d\xi
\eabs {\xi _i}^{2 s - d(\alpha _2-\alpha _1)}
\\
& \leq \sum _{j\in \Omega _i}\int \fy _j^2(\xi )|\widehat f(\xi )|^2\, d\xi
\eabs {\xi _i}^{2 s - d(\alpha _2-\alpha _1)}
\\
& \lesssim \eabs {\xi _i}^{d(\alpha _2-\alpha _1)}\sup _{j\in \Omega _i}
\nm {\fy _j\widehat f}{L^2}^2 \eabs {\xi _i}^{2 s - d(\alpha _2-\alpha _1)}
\\
& = \sup _{j\in \Omega _i}\nm {\fy _j(D)f}{L^2}^2 \eabs {\eta _j}^{2 s}.
\end{aligned}
\end{equation}
Taking the supremum over $i$ we obtain
$$
\nm f{M^{2,\infty}_{\alpha _2,s - d(\alpha _2-\alpha _1)/2}} \lesssim \nm f{M^{2,\infty }_{\alpha _1,s}},
$$
which proves the embedding
\begin{equation}\label{emb2infty}
M^{2,\infty }_{\alpha _1,s}(\rr d)\subseteq M^{2,\infty}_{\alpha _2,s - d(\alpha _2-\alpha _1)/2}(\rr d).
\end{equation}
Next we observe that Young's inequality and \eqref{bapu3} for $\{ \psi _i \}$ gives, for all $i$ and any $p \in [1,\infty]$,
\begin{equation}\label{Lpestimate1}
\| \psi_i(D) f \|_{L^p}
= \left\| \sum_{j \in \Omega_i} \mathscr F^{-1} \left( \psi_i \varphi_j \wh f \right) \right\|_{L^p} \\
\lesssim \sum_{j \in \Omega_i} \left\| \varphi_j (D) f \right\|_{L^p}.
\end{equation}
This gives
\begin{equation}\nonumber
\begin{aligned}
\| f \|_{M_{\alpha_2,s}^{1}} & = \sum_{i} \eabs{\xi_i}^s \| \psi_i(D) f \|_{L^1}
\lesssim \sum_{i} \sum_{j \in \Omega_i} \eabs{\xi_i}^s \left\| \varphi_j (D) f \right\|_{L^1} \\
& \asymp \sum_{i} \sum_{j \in \Omega_i} \eabs{\eta_j}^s \left\| \varphi_j (D) f \right\|_{L^1}
= \sum_{j} \sum_{i \in \Lambda_j} \eabs{\eta_j}^s \left\| \varphi_j (D) f \right\|_{L^1} \\
& \lesssim \| f \|_{M_{\alpha_1,s}^{1}},
\end{aligned}
\end{equation}
which proves the embedding
\begin{equation}\label{emb11}
M^{1}_{\alpha _1,s}(\rr d) \subseteq M^{1}_{\alpha _2,s}(\rr d).
\end{equation}
We also obtain from \eqref{Lpestimate1}
\begin{equation}\nonumber
\begin{aligned}
\| f \|_{M_{\alpha_2,s-d(\alpha_2-\alpha_1)}^{1,\infty}} & = \sup_{i} \eabs{\xi_i}^{s-d(\alpha_2-\alpha_1)} \| \psi_i(D) f \|_{L^1} \\
& \lesssim \sup_{i} \sum_{j \in \Omega_i} \eabs{\xi_i}^{-d(\alpha_2-\alpha_1)} \eabs{\eta_j}^s \left\| \varphi_j (D) f \right\|_{L^1}
\lesssim \| f \|_{M_{\alpha_1,s}^{1,\infty}},
\end{aligned}
\end{equation}
which proves the embedding
\begin{equation}\label{emb1infty}
M^{1,\infty}_{\alpha _1,s}(\rr d) \subseteq M^{1,\infty}_{\alpha _2,s -d(\alpha_2-\alpha_1) }(\rr d).
\end{equation}
Again \eqref{Lpestimate1} gives
\begin{equation}\nonumber
\begin{aligned}
\| f \|_{M_{\alpha_2,s}^{\infty,1}} & = \sum_{i} \eabs{\xi_i}^{s} \| \psi_i(D) f \|_{L^\infty}
\lesssim \sum_{i} \sum_{j \in \Omega_i} \eabs{\eta_j}^s \left\| \varphi_j (D) f \right\|_{L^\infty} \\
& = \sum_{j} \sum_{i \in \Lambda_i} \eabs{\eta_j}^s \left\| \varphi_j (D) f \right\|_{L^\infty}
\lesssim \| f \|_{M_{\alpha_1,s}^{\infty,1}},
\end{aligned}
\end{equation}
which proves the embedding
\begin{equation}\label{embinfty1}
M^{\infty,1}_{\alpha _1,s}(\rr d) \subseteq M^{\infty,1}_{\alpha _2,s }(\rr d).
\end{equation}
Finally \eqref{Lpestimate1} gives
\begin{equation}\nonumber
\begin{aligned}
\| f \|_{M_{\alpha_2,s-d(\alpha_2-\alpha_1)}^{\infty}}
& = \sup_{i} \eabs{\xi_i}^{s -d(\alpha_2-\alpha_1)} \| \psi_i(D) f \|_{L^\infty} \\
& \lesssim \sup_{i} \sum_{j \in \Omega_i} \eabs{\xi_i}^{-d(\alpha_2-\alpha_1)} \eabs{\eta_j}^s \left\| \varphi_j (D) f \right\|_{L^\infty} \\
& \lesssim \| f \|_{M_{\alpha_1,s}^{\infty}},
\end{aligned}
\end{equation}
which proves the embedding
\begin{equation}\label{embinftyinfty}
M^{\infty}_{\alpha _1,s}(\rr d) \subseteq M^{\infty}_{\alpha _2,s - d(\alpha_2-\alpha_1)}(\rr d).
\end{equation}
By Lemma \ref{alphasobolev} we have
\begin{equation}\label{emb22}
M^{2}_{\alpha _1,s}(\rr d) = M^{2}_{\alpha _2,s}(\rr d).
\end{equation}
The result now follws from interpolation between \eqref{emb2infty}, \eqref{emb11}, \eqref{emb1infty}, \eqref{embinfty1}, \eqref{embinftyinfty} and \eqref{emb22}, and duality.
\end{proof}

\section{Sharpness of the embeddings}

The notion of $\alpha$-covering is connected with the metric calculus presented in \cite[Section 18.4]{Horm3}.
Let $0\le \alpha \le 1$, and let $g$ be the Riemannian metric
$$
g_\eta (\xi )=\frac {|\xi |^2}{\eabs \eta ^{2\alpha}}.
$$
If $0<r<1$ then it follows by straight-forward considerations that
$$
g_\eta (\xi -\eta ) \leq r^2\quad \Longrightarrow \quad C^{-1}g_\eta (\zeta )\le g_\xi (\zeta ) \le Cg_\eta (\zeta ),
\quad \zeta \in \rr d,
$$
for some constant $C$ which depends on $r$ only. Hence $g$ is a slowly varying metric
in the sense of \cite[Def.~18.4.1]{Horm3}, and (18.4.2) in \cite{Horm3} is satisfied with $c=r^2$.
The results in \cite{Horm3} gives the following proposition.

\begin{prop}\label{metricpartition}
Let $0\le \alpha \le 1$ and $0<r<1$. The following holds.
\begin{enumerate}
\item[{\rm{(i)}}]
For some sequence $\{ \xi _i\} _{i\in I}\subseteq \rr d$, the balls $B_i =
B(\xi _i,r\eabs  {\xi _i}^\alpha /2)$
constitute an $\alpha$-covering.

\vrum

\item[{\rm{(ii)}}]
There are functions
$\psi _i \in C_c^\infty (\rr d)$, $i \in I$, such that $\supp \psi _i \subseteq B_i$,
$0\le \psi _i\le 1$, $\sum _{i\in I}\psi _i =1$, and for every multiindex $\beta$, there is a
finite constant $C_\beta>0$ such that
\begin{equation}\label{metricderiv1}
\sup _{i\in I}\Big ( \eabs {\xi _i}^{\alpha |\beta |}\nm {\partial ^\beta \psi _i}{L^\infty}\Big ) \le C_\beta \text .
\end{equation}

\vrum

\item[{\rm{(iii)}}]
If $\mathcal Q = \{ B_i \}_{i \in I}$ then $\{ \psi_i \}_{i \in I}$ is a $\mathcal Q$-BAPU.
\end{enumerate}
\end{prop}

\begin{proof}
(i) and (ii) follow immediately from \cite[Lemma 18.4.4]{Horm3} with $\ep < 1/8$.
Therefore, in order to prove (iii) it suffices to show
\begin{equation}\nonumber
\sup_{i \in I} \| \mathscr F \psi_i \|_{L^1} < \infty,
\end{equation}
which is a special case of the following Lemma \ref{fourierLp1}.
\end{proof}

\begin{lem}\label{fourierLp1}
Let $0\le \alpha \le 1$ and suppose $\{ \psi_i \}_{i \in I} \subseteq C_c^\infty(\rr d)$ is a family of functions such that $\supp \psi_i \subseteq B(\xi _i,r\eabs  {\xi _i}^\alpha )$, $i \in I$, for some sequence $\{ \xi_i \}_{i \in I} \subseteq \rr d$ and some $r>0$, and for any multiindex $\beta$ there is $C_\beta>0$ such that
\begin{equation}\label{metricderiv2}
\sup _{i\in I}\Big ( \eabs {\xi _i}^{\alpha |\beta |}\nm {\partial ^\beta \psi _i}{L^\infty}\Big ) \le C_\beta \text .
\end{equation}
Then for $p \in [1,\infty]$ there is a constant $C_p>0$ such that
\begin{equation}\nonumber
\sup_{i \in I} \eabs{\xi_i}^{-d \alpha/p'} \| \mathscr F \psi_i \|_{L^p} \leq C_p.
\end{equation}
\end{lem}

\begin{proof}
Set
$$
\varphi_i(\xi) = \psi_i( \eabs  {\xi _i}^\alpha \xi + \xi_i), \quad i \in I.
$$
Then $\supp \varphi_i \subseteq B(0,r)$ for all $i \in I$, and \eqref{metricderiv2} gives
$\nm {\partial ^\beta \varphi _i}{L^\infty} \leq C_\beta$ for all $i \in I$.
If $p<\infty$ and $n>d/(2 p)$ is an integer then integration by parts gives, for some constants $c_\beta$,
\begin{equation}\nonumber
\begin{aligned}
\| \mathscr F \varphi_i \|_{L^p}^p
& = (2 \pi)^{-dp/2} \int_{\rr d} \eabs{x}^{-2 n p} \left| \int_{\rr d} \varphi_i(\xi) \eabs{x}^{2n} e^{-i x \cdot \xi} d \xi \right|^p dx \\
& = (2 \pi)^{-dp/2} \int_{\rr d} \eabs{x}^{-2 n p} \left| \sum_{|\beta| \leq 2n} c_\beta \int_{\rr d} \partial ^\beta \varphi_i(\xi) e^{-i x \cdot \xi} d \xi \right|^p dx \\
& \lesssim \int_{\rr d} \eabs{x}^{-2 n p} \left( \sum_{|\beta| \leq 2n} \| \partial ^\beta \varphi_i \|_{L^1} \right)^p dx
\lesssim 1
\end{aligned}
\end{equation}
for all $i \in I$.
If $p = \infty$ the observations above give $\| \mathscr F \varphi_i \|_{L^\infty} \leq (2 \pi)^{-d/2} \| \varphi_i \|_{L^1} \lesssim 1$ for all $i \in I$.
The result now follows from $\| \mathscr F \psi_i \|_{L^p} = \eabs{\xi_i}^{d \alpha/p'} \| \mathscr F \varphi_i \|_{L^p}$.
\end{proof}

Given an $\alpha$-covering and an $\alpha$-BAPU according to Proposition \ref{metricpartition},
the next lemma says that
we may adjoin a sequence of balls to the covering, and modify the BAPU accordingly,
without destroying the $\alpha$-covering and the $\alpha$-BAPU properties.
A function indexed by the new index set
equals one on a ball of radius proportional to $\eabs{\xi _j}^\alpha$ where $\xi _j$ is the center of the support of the function.
This will be useful in the proofs of the forthcoming sharpness results Propositions \ref{counterp1} and \ref{counterp2}.

\begin{lem}\label{extfamBi}
Let $0\le \alpha \le 1$, $0<r<1$, and let $\{ B_i \} _{i\in I}$ and $\{  \psi _i \} _{i\in I}$ be
as in Proposition \ref{metricpartition}.
Let $J$ be a countable index set such that $I\cap J=\emptyset$,
and let $\{ B_j \} _{j\in J}$ be balls such that
$B_j = B(\xi _j,r\eabs {\xi _j}^\alpha /2)$ where $\xi _j\in \rr d$ for $j\in J$,
and $B_j\cap B_k =\emptyset$, when $j,k\in J$ and $j\neq k$.

Then there are functions $\fy _i \in C_c^\infty (\rr d)$, $i\in I\cup J$, such that the following is true:

\begin{enumerate}
\item[{\rm{(i)}}] $0\le \fy _i\le 1$, $\supp \fy _i \subseteq B_i$ when $i\in I\cup  J$;

\vrum

\item[{\rm{(ii)}}] $\fy _j =1$ on $B(\xi _j,r \eabs {\xi _j }^\alpha /4)$ for $j\in J$;

\vrum

\item[{\rm{(iii)}}] $\{ \fy _i \}_{i \in I \cup J}$ is an $\alpha$-BAPU, and for each multiindex $\beta$ there exists $C_\beta>0$ such that
\begin{equation}\label{fyests1}
\sup_{i \in I \cup J} \Big ( \eabs {\xi _i}^{\alpha |\beta |}\nm {\partial ^\beta \fy _i}{L^\infty}\Big ) \le C_\beta.
\end{equation}
\end{enumerate}
\end{lem}

\begin{proof}
Let $\varphi \in C_c^\infty (\rr d)$, $0 \leq \varphi \leq 1$, $\supp \varphi \subseteq B(0,r/2)$ and
$\varphi(\xi) = 1$ for $\xi \in B(0,r/4)$.
We set
\begin{alignat*}{3}
\varphi _j(\xi ) &= \varphi (\eabs {\xi _j}^{-\alpha}(\xi -\xi _j))& \quad &\text{for}&\quad j&\in J
\intertext{and}
\varphi _i(\xi ) &= \psi _i(\xi )\prod _{j\in J}(1-\varphi _j(\xi ))& \quad &\text{for}&\quad i&\in I.
\end{alignat*}
Then properties (i) and (ii) are satisfied.
The estimate $\sup_{j \in J} \eabs {\xi _j}^{\alpha |\beta |}\nm {\partial ^\beta \fy _j}{L^\infty} < C_\beta$ for any multiindex $\beta$ follows immediately.
These estimates combined with \eqref{metricderiv1} and straightforward considerations give
$\sup_{i \in I} \eabs {\xi _i}^{\alpha |\beta |}\nm {\partial ^\beta \fy _i}{L^\infty} < C_\beta$ for all multiindices $\beta$.
Thus \eqref{fyests1} holds for all multiindices $\beta$.
Likewise one can easily verify
\begin{equation}\nonumber
\sum _{i \in I\cup J}\fy _i (\xi) = 1 \quad \forall \xi \in \rr d,
\end{equation}
as well as the fact that $\{ B_i, B_j \}_{i \in I, j \in J}$ is an admissible $\alpha$-covering.
To prove (iii) it thus suffices to observe that $\sup_{j \in J} \| \mathscr F \varphi_j \|_{L^1} < \infty$ follows from $\| \mathscr F \varphi_j \|_{L^1} = \| \mathscr F \varphi \|_{L^1}$, and that $\sup_{i \in I} \| \mathscr F \varphi_i \|_{L^1} < \infty$
follows from \eqref{fyests1} and Lemma \ref{fourierLp1}.
\end{proof}

We are now in a position to prove two results which show that the embeddings \eqref{newembedding1} in Theorem
\ref{alphaembedding} are optimal, in most cases. This is a consequence of the following Propositions \ref{counterp1} and \ref{counterp2}.

\begin{prop}\label{counterp1}
If $p,q\in [1,\infty]$, $0\le \alpha _1\le \alpha _2 \le 1$ and $t,s \in \ro$ then
\begin{equation}\nonumber
M^{p,q}_{\alpha _1,s} \subseteq M^{p,q}_{\alpha _2,t} \quad \Longrightarrow \quad t \leq s+d(\alpha _2-\alpha _1)\Big ( \frac 1q-\frac 1{p'}\Big ).
\end{equation}
\end{prop}

\begin{proof}
We prove the result by showing that the assumption
$$
\ep := t - s - d(\alpha _2-\alpha _1)(1/q-1/{p'}) > 0
$$
implies that
\begin{equation}\label{contradiction2}
M^{p,q}_{\alpha _1,s} \subseteq M^{p,q}_{\alpha _2,t}
\end{equation}
cannot hold.

Let $\{ \fy _j \}_{j \in J}$ be an $\alpha_1$-BAPU constructed according to Proposition \ref{metricpartition},
and let $\{ \psi _i \}$ be an $\alpha_2$-BAPU constructed according to Proposition \ref{metricpartition} and modified according to Lemma \ref{extfamBi}. Then there exists an infinite index set $I$ such that the following is true for some $r>0$:

\begin{enumerate}
\item[{\rm{(i)}}] If $i_1,i_2\in I$ and $i_1\neq i_2$, then $\supp \psi _{i_1}\cap \supp \psi _{i_2}=\emptyset$;

\vrum

\item[{\rm{(ii)}}] $\psi _i(\xi )=1$ on $B_i=B(\xi _i, r \eabs {\xi _i}^{\alpha _2} )$, $\xi_i \in \rr d$, $i\in I$.
\end{enumerate}

Let $\vartheta \in C_c^\infty (\rr d)$ satisfy $0\le \vartheta \le 1$, $\supp \vartheta \subseteq B(0,r)$ and
$\vartheta (\xi )=1$ when $\xi \in B(0,r/2)$, and define $\vartheta _i (\xi)=\vartheta (\eabs {\xi _i}^{-\alpha _2}(\xi -\xi _i ))$.
Then $\psi _i=1$ in $\supp \vartheta _i$.
Let $I' \subseteq I$ be any finite subset, let $\{t_i\} _{i\in I'}$ be a sequence of nonnegative numbers, and set
$$
\widehat f(\xi ) =\sum _{i\in I'} t_i \vartheta _i(\xi ) \in C_c^\infty(\rr d).
$$
Let $q<\infty$.
It follows from our choice of $\vartheta _i$ that
\begin{equation}\label{alpha2normest1}
\begin{aligned}
\nm f{M^{p,q}_{\alpha _2,t}} & \ge \Big (\sum _{i\in I'} \big ( \eabs {\xi _i}^{t}
\nm {\psi _i(D)f}{L^p}\big )^q\Big )^{1/q}
\\
& = \Big (\sum _{i\in I'} \big ( \eabs {\xi _i}^{t}t_i\, \nm {\widehat \vartheta _i}{L^p}\big )^q\Big )^{1/q}
\asymp  \Big (\sum _{i\in I'} \big ( t_i\eabs {\xi _i}^{t+d\alpha _2/p'}\big )^q\Big )^{1/q}.
\end{aligned}
\end{equation}
Next we estimate $\nm f{M^{p,q}_{\alpha _1,s}}$.
Set
\begin{equation}\nonumber
\begin{aligned}
J_i & = \sets {j\in J}{\supp \fy _j\cap B_i\neq \emptyset}, \quad i \in I', \\
I'_j & = \sets {i \in I'}{\supp \fy _j\cap B_i\neq \emptyset}, \quad j \in J.
\end{aligned}
\end{equation}
By Lemma \ref{countinglemma},
\begin{equation}\nonumber
\begin{aligned}
|J_i| & \lesssim \eabs {\xi _i}^{d(\alpha _2-\alpha _1)}, \quad & i \in I', \\
|I'_j| & \lesssim 1, \quad & j \in J.
\end{aligned}
\end{equation}
Denoting the center of the ball in which $\varphi_j$ is supported by $\eta_j \in \rr d$,
this gives, using H\"older's and Young's inequalities, Lemma \ref{countinglemma} and Lemma \ref{fourierLp1},
\begin{equation}\label{alpha1normest1}
\begin{aligned}
\nm f{M^{p,q}_{\alpha _1,s}} &
= \left(  \sum_{j \in J} \eabs{\eta_j}^{sq} \left\| \sum_{i\in I_j'} t_i \mathscr F^{-1} \left( \varphi_j \vartheta _i \right) \right\|_{L^p}^q \right)^{1/q} \\
& \lesssim \left(  \sum_{j \in J} \eabs{\eta_j}^{sq} \sum_{i\in I_j'} t_i^q \| \mathscr F^{-1} \left( \varphi_j \vartheta _i \right) \|_{L^p}^q \right)^{1/q} \\
& \lesssim \left(  \sum_{i\in I'} \sum_{j \in J_i} \eabs{\eta_j}^{sq}  t_i^q \| \mathscr F^{-1} \vartheta _i \|_{L^1}^q \| \mathscr F^{-1} \varphi_j \|_{L^p}^q \right)^{1/q} \\
& \lesssim \left(  \sum_{i\in I'} \sum_{j \in J_i} \eabs{\eta_j}^{sq}  t_i^q \| \mathscr F^{-1} \varphi_j \|_{L^p}^q \right)^{1/q} \\
& \lesssim \left(  \sum_{i\in I'} \sum_{j \in J_i} \eabs{\xi_i}^{sq + d \alpha_1 q/p'} t_i^q \right)^{1/q} \\
& \lesssim \left(  \sum_{i\in I'} \left( t_i \eabs{\xi_i}^{s + d(\alpha _2-\alpha _1)/q + d \alpha_1 /p'} \right)^q  \right)^{1/q}.
\end{aligned}
\end{equation}
We may assume that $I = \noo$.
Since $|\xi_i| \rightarrow \infty$ as $i \rightarrow \infty$, we may assume that
$\eabs{\xi_i} \geq \eabs{i}^{\frac{2}{\ep q}}$, by passing to a subsequence if necessary.
If we set
$$
t_i := \eabs{i}^{-\frac{2}{q}} \eabs{\xi_i}^{- s - d(\alpha_2-\alpha_1)/q - d \alpha_1/p'}
$$
then \eqref{alpha2normest1} and \eqref{alpha1normest1} give a contradiction to \eqref{contradiction2}, as $|I'|$ is made arbitrarily large. This proves the result when $q<\infty$.
The case $q=\infty$ is settled with slight modifications of the same proof.
\end{proof}

\begin{prop}\label{counterp2}
If $p,q\in [1,\infty]$, $0\le \alpha _1\le \alpha _2 \le 1$ and $t,s \in \ro$ then
\begin{equation}\nonumber
M^{p,q}_{\alpha _1,s} \subseteq M^{p,q}_{\alpha _2,t} \quad \Longrightarrow \quad t \leq s.
\end{equation}
\end{prop}

\begin{proof}
We show that $t > s$ implies that \eqref{contradiction2}
does not hold.

Let $\{ \fy _j \}_{j \in J}$, $\{ \psi _i \}$ and $I$ be as in the proof of Proposition \ref{counterp1} and let $\vartheta _i=\vartheta (\xi -\xi _i )\in C_c^\infty (\rr d)$,
where $\vartheta \in C_c^\infty (\rr d)$, $\supp \vartheta \subseteq B(0,r)$ is the same as in the proof of Proposition \ref{counterp1}. Let $f$ be given by
$$
\widehat f(\xi ) =\sum _{i\in I'} t_i \vartheta _i(\xi ) \in C_c^\infty(\rr d)
$$
for some suitable sequence $\{ t_i \}_{i \in I'}$ where $I' \subseteq I$ is finite. Let $q<\infty$.
We have
\begin{multline}\label{alpha2normest2}
\nm f{M^{p,q}_{\alpha _2,t}} \ge \Big (\sum _{i\in I'} \big ( \eabs {\xi _i}^{t}
\nm {\psi _i(D)f}{L^p}\big )^q\Big )^{1/q}
\\[1ex]
= \Big (\sum _{i\in I'} \big ( \eabs {\xi _i}^{t}t_i\, \nm {\widehat \vartheta _i}{L^p}\big )^q\Big )^{1/q}
\asymp  \Big (\sum _{i\in I'} \big ( t_i\eabs {\xi _i}^{t}\big )^q\Big )^{1/q}.
\end{multline}
In order to estimate $\nm f{M^{p,q}_{\alpha _1,s}}$ we set
\begin{equation}\nonumber
\begin{aligned}
J_i & = \sets {j\in J}{\supp \fy _j\cap B(\xi_i,r) \neq \emptyset}, \quad i \in I', \\
I'_j & = \sets {i \in I'}{\supp \fy _j\cap B(\xi_i,r) \neq \emptyset}, \quad j \in J.
\end{aligned}
\end{equation}
As in the proof of Lemma \ref{countinglemma} it follows that
$$
\sup_{i \in I'} | J_i | < \infty, \quad \sup_{j \in J} | I_j' | < \infty, \quad \mbox{and} \quad
\eabs{\xi_i}  \asymp \eabs {\eta _j} \quad \mbox{when} \quad j \in J_i.
$$
As in the estimate \eqref{alpha1normest1} this gives, again using H\"older's and Young's inequalities and Lemma \ref{fourierLp1},
\begin{equation}\label{alpha1normest2}
\begin{aligned}
\nm f{M^{p,q}_{\alpha _1,s}} &
= \left(  \sum_{j \in J} \eabs{\eta_j}^{sq} \left\| \sum_{i\in I_j'} t_i \mathscr F^{-1} \left( \varphi_j \vartheta _i \right) \right\|_{L^p}^q \right)^{1/q} \\
& \lesssim \left(  \sum_{j \in J} \eabs{\eta_j}^{sq} \sum_{i\in I_j'} t_i^q \| \mathscr F^{-1} \left( \varphi_j \vartheta _i \right) \|_{L^p}^q \right)^{1/q} \\
& \lesssim \left(  \sum_{i\in I'} \sum_{j \in J_i} \eabs{\xi_i}^{sq} t_i^q \| \mathscr F^{-1} \vartheta _i \|_{L^p}^q \| \mathscr F^{-1} \varphi_j \|_{L^1}^q \right)^{1/q} \\
& \lesssim \left(  \sum_{i\in I'}  \eabs{\xi_i}^{sq}  t_i^q \right)^{1/q}.
\end{aligned}
\end{equation}
As before \eqref{alpha2normest2} and \eqref{alpha1normest2} give a contradiction to \eqref{contradiction2}.
The case $q=\infty$ follows in the same manner.
\end{proof}

A combination of \eqref{indices1}, Propositions \ref{counterp1} and \ref{counterp2}, and duality give the
earlier mentioned optimality result concerning Theorem \ref{alphaembedding}.

\begin{cor}\label{sharpnessresult}
Let $p,q \in [1,\infty]$, $s \in \ro$ and $0 \leq \alpha_1 \leq \alpha_2 \leq 1$. \\
If $1/p \leq \max(1/2,1/q)$ then
\begin{equation}\nonumber
M^{p,q}_{\alpha _1,s} \subseteq M^{p,q}_{\alpha _2,t} \quad \Longrightarrow \quad t \leq s + d(\alpha_2-\alpha_1) \theta_2(p,q).
\end{equation}
If $1/p \geq \min(1/2,1/q)$ then
\begin{equation}\nonumber
M^{p,q}_{\alpha _2,t} \subseteq M^{p,q}_{\alpha _1,s} \quad \Longrightarrow \quad t \geq s + d(\alpha_2-\alpha_1) \theta_1(p,q).
\end{equation}
\end{cor}

\subsection*{Acknowledgements}

Part of the content of this article was presented at the
International Conference on Partial Differential Equations and Applications,
Sofia, Bulgaria, September 14--17, 2011, in honor of Professor Petar Popivanov on the occasion of his 65th anniversary.
The second named author thanks the organizers of the conference
Georgi Boyadjiev, Todor Gramchev, Nikolay Kutev and Alessandro Oliaro,
for the interesting conference, and for the invitation to give a talk.

Furthermore, an early version of the content was also presented at the conference
From Abstract to Computational Harmonic Analysis, June 13--19, 2011, Strobl, Austria.
We thank the organizers Karlheinz Gr\"ochenig and Thomas Strohmer, as well as Hans G. Feichtinger, for the opportunity to present the paper and for the stimulating conference.



\begin{thebibliography}{DD}

\bibitem{Benedek1}
\textsc{A. Benedek and R. Panzone.}
The space $L^p$, with mixed norm.
\textit{Duke Math. J.} \textbf{28}, (1961), 301--24.

\bibitem{Bergh1}
\textsc{J. Bergh and J. L\"ofstr\"om.}
Interpolation Spaces, An Introduction.
Springer-Verlag, Berlin Heidelberg New York, 1976.

\bibitem{Borup1}
\textsc{L. Borup and M. Nielsen.}
Banach frames for multivariate $\alpha$-modulation spaces.
\textit{J. Math. Anal. Appl.} \textbf{321}, (2006), 880--895.

\bibitem{Feichtinger1}
\textsc{H. G. Feichtinger}
Modulation spaces on locally compact
abelian groups.
Technical report, {University of
Vienna}, Vienna, 1983; also in: M. Krishna, R. Radha,
S. Thangavelu (Eds) Wavelets and their applications, Allied
Publishers Private Limited, New Delhi Mumbai Kolkata Chennai Hagpur
Ahmedabad Bangalore Hyderbad Lucknow, 2003, pp. 99--140.

\bibitem{Feichtinger2}
\textsc{H. G. Feichtinger and P. Gr\"obner.}
Banach spaces of distributions defined by decomposition methods, I.
\textit{Math. Nachr.} \textbf{123}, (1985), 97--120.

\bibitem{Feichtinger3}
\textsc{H. G. Feichtinger.}
Banach spaces of distributions defined by decomposition mehtods II.
\textit{Math. Nachr.} \textbf{132}, (1987), 207--237.

\bibitem{Fornasier1}
\textsc{M. Fornasier.}
Banach frames for $\alpha$-modulation spaces.
\textit{Appl. Comput. Harmon. Anal.} \textbf{22}, (2007), 157--175.

\bibitem{Grobner1}
\textsc{P. Gr{\" o}bner.}
Banachr\" aume glatter Funktionen und zerlegungsmethoden.
PhD Thesis, University of Vienna, 1992.

\bibitem{Grochenig1}
\textsc{K. Gr{\" o}chenig}
Foundations of Time-Frequency Analysis, Birkh{\"a}user, Boston, 2001.

\bibitem{Hanwang1}
\textsc{J. Han and B. Wang.}
$\alpha$-modulation spaces (I), arXiv:1108.0460v2 [math.FA], 2011.

\bibitem{Horm1}
\textsc{L. H{\"o}rmander.}
The Analysis of Linear Partial Differential Operators, vol I.
Springer-Verlag, Berlin Heidelberg New York Tokyo, 1990.

\bibitem{Horm3}
\textsc{L. H{\"o}rmander.}
The Analysis of Linear Partial Differential Operators, vol III.
Springer-Verlag, Berlin Heidelberg New York Tokyo, 1994.

\bibitem{Okoudjou1}
\textsc{K. A. Okoudjou}
Embeddings of some classical Banach spaces into modulation spaces.
\textit{Proc. Amer. Math. Soc.} \textbf{132}, (2004), 1639--1647.

\bibitem{Sugimoto1}
\textsc{M. Sugimoto and N. Tomita.}
The dilation property of modulation spaces and their inclusion relation with Besov spaces.
\textit{J. Funct. Anal.} \textbf{248}, (2007), 79--106.

\bibitem{Toft1}
\textsc{J. Toft.}
Continuity properties for modulation spaces, with applications to pseudo-differential calculus -- I.
\textit{J. Funct. Anal.} \textbf{207} (2004), 399--429.

\bibitem{Toft2}
\textsc{J. Toft.}
Continuity properties for modulation spaces, with applications to pseudo-differential calculus -- II.
\textit{Ann. Glob. Anal. Geom.} \textbf{26}, (2004), 73--106.

\bibitem{Wang1}
\textsc{B. Wang and C. Huang.}
Frequency-uniform decomposition method
for the generalized BO, KdV and NLS equations.
\textit{J. Differential Equations.} \textbf{239}, (2007), 213--250.

\end{thebibliography}
\end{document}